\documentclass[10pt,oneside,fleqn]{article}
\usepackage[a4paper]{}
\usepackage{amssymb,fleqn,amsmath, amsthm, amscd, amsfonts}
\usepackage{color}
\usepackage{hyperref}
\usepackage{geometry}

\def\c{{\bf {c}}}
\usepackage{amsmath, amsthm, amscd, color, amsfonts}
\usepackage{graphicx}
\usepackage{epsfig}
\usepackage{hyperref}
\newtheorem{thm}{Theorem}[section]

\newtheorem{lemma}[thm]{Lemma}

\theoremstyle{definition}

\title{General matrix transform method for the Riesz space fractional advection-dispersion equations}
\author{Abdollah Borhanifar\footnote{Corresponding author (Email): borhani@uma.ac.ir; (Tel):+984531505904; (Fax):+984531514702.}, Sohrab Valizadeh
 \\
 {\footnotesize\em Department of Mathematics and applications, Faculty of Sciences, University of Mohaghegh Ardabili, IR56199-179, Ardabil, Iran}\\
 }
\date{\today}
\numberwithin{equation}{section}
\begin{document}
\maketitle
\newtheorem{Definition}{Definition}[section]
\newtheorem{theorem}{Theorem}[section]
\newtheorem{example}{Example}[section]
\newtheorem{proposition}{Proposition}
\begin{abstract}
In this paper, a mixed high order finite difference scheme-Pad\'{e} approximation method is applied to obtain numerical solution of the Riesz space fractional advection-dispersion equation(RSFADE). This method is based on the high order finite difference scheme that derived from fractional centered difference and Pad\'{e} approximation method for space and time integration, respectively. The stability analysis of the proposed method is discussed via theoretical matrix analysis. Numerical experiments are presented to confirm the theoretical results of the proposed method.
\end{abstract}
\section{Introduction}
\label{sec:Introduction}
Fractional partial differential equations have been increasingly attracting interest over the last two decades because of its demonstrated applications in numerous apparently diverse and widespread fields of science and engineering, including viscoelasticity, fluid mechanics \cite{Kilbas2006}, finance \cite{Scalas2000376}, medical imaging \cite{Yu2013}, analytical chemistry, fractional multi-polar \cite{Hilfer2000}, formulating physical, chemical and biological sciences \cite{Magin2006}, and hydrology \cite{Schumer200169,Schumer20031022}.\\

Various explanations were offered for fractional-order derivatives and integrals can be mentioned such as the Riemann-Liouville, Caputo, Gr\"{u}nwald-Letnikov and other approaches \cite{Borhanifar2012433,Borhanifar2013427,Borhanifar2015466}. The Riesz fractional derivative in the article dealing with this type of derivative which combined on the left and right Riemann-Liouville derivatives with coefficients with respect to the corresponding order \cite{Valizadeh202018}. This type derivative is showcased in the fractional advection-dispersion equation that is applied to model the transport of passive tracers carried by fluid flow in a porous medium.\\

According to importance of subject, a number of mathematicians and researchers have done theoretical studies in the field of fractional calculus. Abbas and Ragusa \cite{Abbas20201} scrutinized the solvability of Langevin equations with two Hadamard fractional derivatives via Mittag-Leffler functions. The regularity properties of solutions of elliptic, parabolic and ultraparabolic equations of second order with discontinuous coefficients have been discussed in depth \cite{Ragusa201294}. Ragusa and Scapellato \cite{Ragusa201751} obtained regularity results for solutions of partial differential equations of parabolic type.  Valizadeh et al. \cite{Valizadeh202012} surveyed optimal feedback control for fractional semilinear integro-differential equations in an arbitrary Banach space.  Review, analyze and solve these equations is a concern for researchers. To obtain a closed form solution to these problems is not always possible and in many cases impossible. Therefore, approximate methods must be used to resolve these issues.\\

Many researchers have presented algorithmic approaches to solve the Riesz fractional advection-dispersion equations (RFADE). Shen et al. \cite{Shen2008850} examined fundamental solution and numerical solution for the Riesz fractional advection-dispersion equation and also \cite{Shen2011383} proposed explicit and implicit difference approximations for the space-time Riesz-Caputo fractional advection-diffusion equation with initial value and boundary conditions in a finite domain and by using mathematical induction proved that the implicit difference approximation is unconditionally stable and convergent, but the explicit difference approximation is conditionally stable and convergent and further the Richardson extrapolation method has been handled to remedy the shortcomings of this method. Yang et al. \cite{Yang2010200} considered L2-approximation, shifted Gr\"{u}nwald approximation and matrix transform methods for fractional partial differential equations with Riesz space fractional derivatives. Ding and Zhang \cite{Ding20121135} applied matrix transform method for Riesz fractional diffusion equation(RFDE) and RFADE by powering to half-order of fractional derivative of Toeplitz matrix corresponding to the second derivative also \cite{Zhang2014266} used improved matrix transform method for the Riesz space fractional reaction dispersion equation. \c{C}elik and Duman \cite{Celik20121743} utilized a mixed fractional centered difference that defined by Ortigueira \cite{Ortigueira20061} for the discreting of the Riesz fractional derivative and Crank-Nicolson method for solving the fractional diffusion equation with the Riesz fractional derivative. Chen et al. \cite{Chen201322} examined superlinearly convergent algorithms for the two-dimensional space-time Caputo-Riesz fractional diffusion equation.
Popolizio \cite{Popolizio20131975} discussed the discretization of Riesz derivatives by fractional centered difference schemes and described the coefficients of this discretization scheme and then obtained the explicit expression of the resulting discretization matrix. Rahman et al. \cite{Rahman2014264} estimated the RFADE with respect to the space variable via improved matrix transform method and after $[3,1]$ Pad\'{e} approximation used  construct the numerical computation of exponential matrix in the analytical form of the out coming ordinary differential equation(ODE).\\

The aim of the current study is to provide a generalized matrix transform method for solving the Riesz space fractional advection-dispersion equation. In order to, we consider the following fractional partial differential equations with the Riesz space fractional derivatives:
\begin{eqnarray}\label{eqnarray1804090938}
\frac{\partial u(x,t)}{\partial
t}=\mathcal{K}_{\alpha}\frac{\partial^{\alpha}u(x,t)}{\partial
|x|^{\alpha}}+\mathcal{K}_{\beta}\frac{\partial^{\beta}u(x,t)}{\partial
|x|^{\beta}},\quad 0< x < L,\quad 0\leq t \leq T,
\end{eqnarray}
subject to the initial value and zero Dirichlet boundary conditions
given by
\begin{eqnarray}\label{eqnarray1804090939}
u(x,0)=\psi(x),\quad 0\leq x \leq L,
\end{eqnarray}
\begin{eqnarray}\label{eqnarray1804090940}
u(0,t)=u(L,t)=0,\quad 0\leq t \leq T
\end{eqnarray}
where $u$ is a solute concentration, $\frac{\partial^{\alpha}}{\partial
|x|^{\alpha}}$ for $1<\alpha\leq2$ and $\frac{\partial^{\beta}}{\partial
|x|^{\beta}}$ for $0<\beta<1$ are Riesz space fractional operators on a finite domain $[0,L]$ and $\mathcal{K}_{\alpha}>0$  and
$\mathcal{K}_{\beta}\geq0$ represent the dispersion coefficient and the average fluid velocity, respectively.\\

The layout of this paper is organized as follows: In Section 2, preliminaries, basic definitions, essential theorems and lemmas are presented that need in the proceeding of this paper, Section 3 contains the introduce of numerical method via Pad\'{e} approximate and fractional centered scheme in novel form, In section 4, the stability of the proposed method for the Riesz fractional advection-dispersion equation is discussed, and the accuracy and efficiency of the iterative scheme are checked by numerical experiments in Sections 5. Finally, we end up this paper by conclusion in Section 6.
\section{Preliminaries}
\label{sec2}
\noindent
In this section, we consider some important definitions and lemmas that will be necessary for encouraging process the aims of paper.
\begin{Definition}
The left- and right-sided Riemann-Liouville fractional derivatives of order $\nu$
of $f(x)$ that be a continuous and necessary function on $[0,L]$ are defined respectively as \cite{Podlubny1999},
\begin{eqnarray}\label{eqnarray1804090958}
_{0} D_{x}^{\nu}f(x)=\frac{1}{\Gamma(n-\nu)}\frac{d^{n}}{d x^{n}}\int_{0}^{x}\frac{f(\xi)}{(x-\xi)^{\nu-n+1}}d\xi,
\end{eqnarray}
\begin{eqnarray}\label{eqnarray1804090959}
_{x} D_{L}^{\nu}f(x)=\frac{(-1)^n}{\Gamma(n-\nu)}\frac{d^{n}}{d x^{n}}\int_{x}^{L}\frac{f(\xi)}{(\xi-x)^{\nu-n+1}}d\xi,
\end{eqnarray}
where $n-1 < \nu \leq n$, $n\in \mathbb{N}$ and $n$ is the smallest integer greater than $\nu$.
\end{Definition}
\begin{Definition}
The Riesz fractional derivatives of order $\nu$ of $f(x)$ on a finite interval $[0,L]$ is defined as \cite{Gorenflo1999231},
\begin{eqnarray}\label{eqnarray1804091002}
\frac{d^{\nu}f(x)}{d
|x|^{\nu}}=-c_{\nu}\{_{0}D_{x}^{\nu}f(x)+_{x}D_{L}^{\nu}f(x)\},
\end{eqnarray}
where
$c_{\nu}=\frac{1}{2cos(\frac{\pi\nu}{2})}$, $n-1< \nu \leq n$ and $\nu\neq1$.
\end{Definition}
Here we consider the approximation with step $\mathit{h}$ of the Riesz fractional derivative that obtained by calculating the appropriate coefficients for the fractional central difference by applying Fourier transform \cite{Ding2015218}\\
\begin{eqnarray}\label{eqnarray1804091007}
\frac{d^{\nu}u(x)}{d
|x|^{\nu}}=-h^{-\nu}\sum_{s\in\textbf{Z}}\vartheta_{s,p}^{(\nu)}\mathcal{H}_{s}^{(\nu)}u(x)+\mathcal{O}(h^{p}),\quad n-1< \nu \leq n \quad and \quad \nu\neq1,
\end{eqnarray}
where
\begin{eqnarray}\label{eqnarray1804091008}
\mathcal{H}_{s}^{(\nu)}u(x)=\sum_{k=-\infty}^{\infty}\omega_{k}^{(\nu)}u(x-(k+s)h),
\end{eqnarray}
and all coefficients $\omega_{k}^{(\nu)}$ are defined by
\begin{eqnarray}\label{eqnarray1804091009}
\omega_{k}^{(\nu)}=\frac{(-1)^{k}\Gamma(\nu+1)}{\Gamma(\frac{\nu}{2}-k+1)\Gamma(\frac{\nu}{2}+k+1)}, \quad k=0,\pm 1,\pm 2,...,
\end{eqnarray}
We survey the properties of the coefficients $\omega_{k}^{(\nu)}$ that are appeared at the approximate formula for Riesz fractional derivatives.
\begin{lemma}\label{Thm99}
    \cite{Celik20121743}The coefficients  $\omega_{k}^{(\nu)}$ for $k\in \mathbb{Z}$ in (\ref{eqnarray1804091009}) satisfy:\\
(a) $\omega_{0}^{(\nu)}\geq0$, $\omega_{-k}^{(\nu)}=\omega_{k}^{(\nu)}\leq 0$ for all $\mid k \mid \geq 1$,\\
(b) $\sum_{k=-\infty}^{\infty}\omega_{k}^{(\nu)}=0$,\\
(c) For any positive integer $n$ and $m$ with $n<m$, we have $\sum_{k=-m+n}^{n}\omega_{k}^{(\nu)}>0$.
\end{lemma}

By considering references \cite{Ding2015221} and \cite{Ding2015218}, the coefficients for various indices $p=2,4,6,8,10$ and $12$ are regularly given below\\
for $p=2$,
\begin{eqnarray*}
\vartheta_{0,2}^{(\nu)}=1,
\end{eqnarray*}
for $p=4$,
\begin{eqnarray*}
\vartheta_{-1,4}^{(\nu)}=\vartheta_{1,4}^{(\nu)}=-\frac{\nu}{24}, \quad \vartheta_{0,4}^{(\nu)}=\frac{\nu}{12}+1,
\end{eqnarray*}
for $p=6$,
\begin{eqnarray*}
\vartheta_{-2,6}^{(\nu)}=\vartheta_{2,6}^{(\nu)}=(\frac{\nu}{1152}+\frac{11}{2880})\nu, \quad \vartheta_{-1,6}^{(\nu)}=\vartheta_{1,6}^{(\nu)}=-(\frac{\nu}{288}+\frac{41}{720})\nu,
\end{eqnarray*}
\begin{eqnarray*}
\vartheta_{0,6}^{(\nu)}=\frac{\nu^2}{192}+\frac{17\nu}{160}+1
\end{eqnarray*}
for $p=8$,
\begin{eqnarray*}
\vartheta_{-3,8}^{(\nu)}=\vartheta_{3,8}^{(\nu)}=-(\frac{\nu^2}{82944}+\frac{11\nu}{69120}+\frac{191}{362880})\nu,
\end{eqnarray*}
\begin{eqnarray*} \vartheta_{-2,8}^{(\nu)}=\vartheta_{2,8}^{(\nu)}=(\frac{\nu^2}{13824}+\frac{7\nu}{3840}+\frac{211}{30240})\nu,
\end{eqnarray*}
\begin{eqnarray*}
\vartheta_{-1,8}^{(\nu)}=\vartheta_{1,8}^{(\nu)}=-(\frac{5\nu^2}{27648}+\frac{3\nu}{512}+\frac{7843}{120960})\nu, \quad \vartheta_{0,8}^{(\nu)}=\frac{5\nu^3}{20736}+\frac{29\nu^2}{3456}+\frac{5297\nu}{45360}+1,
\end{eqnarray*}
for $p=10$,
\begin{eqnarray*}
\vartheta_{-4,10}^{(\nu)}=\vartheta_{4,10}^{(\nu)}=(\frac{\nu^3}{7962624}+\frac{11\nu^2}{3317760}+\frac{10181\nu}{348364800}+\frac{2497}{29030400})\nu, \quad
\end{eqnarray*}
\begin{eqnarray*}
\vartheta_{-3,10}^{(\nu)}=\vartheta_{3,10}^{(\nu)}=-(\frac{\nu^3}{995328}+\frac{\nu^2}{25920}+\frac{17111\nu}{43545600}+\frac{1469}{1209600})\nu,
\end{eqnarray*}
\begin{eqnarray*}
\vartheta_{-2,10}^{(\nu)}=\vartheta_{2,10}^{(\nu)}=(\frac{7\nu^3}{1990656}+\frac{137\nu^2}{829440}+\frac{32861\nu}{12441600}+\frac{68119}{7257600})\nu, \quad
\end{eqnarray*}
\begin{eqnarray*}
\vartheta_{-1,10}^{(\nu)}=\vartheta_{1,10}^{(\nu)}=-(\frac{7\nu^3}{995328}+\frac{19\nu^2}{51840}+\frac{46631\nu}{6220800}+\frac{252769}{3628800})\nu,
\end{eqnarray*}
\begin{eqnarray*}
\vartheta_{0,10}^{(\nu)}=\frac{35\nu^4}{3981312}+\frac{157\nu^3}{331776}+\frac{51941\nu^2}{4976640}+\frac{118829\nu}{967680}+1,
\end{eqnarray*}
for $p=12$,
\begin{eqnarray*}
\vartheta_{-5,12}^{(\nu)}=\vartheta_{5,12}^{(\nu)}=-(\frac{\nu^4}{955514880}+\frac{11\nu^3}{238878720}+\frac{6361\nu^2}{8360755200}+\frac{11693\nu}{2090188800}+\frac{14797}{958003200})\nu, \quad
\end{eqnarray*}
\begin{eqnarray*}
\vartheta_{-4,12}^{(\nu)}=\vartheta_{4,12}^{(\nu)}=(\frac{\nu^4}{95551488}+\frac{7\nu^3}{11943936}+\frac{9133\nu^2}{836075520}+\frac{5563\nu}{65318400}+\frac{230371}{958003200})\nu, \quad
\end{eqnarray*}
\begin{eqnarray*}
\vartheta_{-3,12}^{(\nu)}=\vartheta_{3,12}^{(\nu)}=-(\frac{\nu^4}{21233664}+\frac{49\nu^3}{15925248}+\frac{13529\nu^2}{185794560}+\frac{449171\nu}{696729600}+\frac{203257}{106444800})\nu, \quad
\end{eqnarray*}
\begin{eqnarray*}
\vartheta_{-2,12}^{(\nu)}=\vartheta_{2,12}^{(\nu)}=(\frac{\nu^4}{7962624}+\frac{\nu^3}{110592}+\frac{17869\nu^2}{69672960}+\frac{24041\nu}{7257600}+\frac{299093}{26611200})\nu, \quad
\end{eqnarray*}
\begin{eqnarray*}
\vartheta_{-1,12}^{(\nu)}=\vartheta_{1,12}^{(\nu)}=-(\frac{7\nu^4}{31850496}+\frac{133\nu^3}{7962624}+\frac{20953\nu^2}{39813120}+\frac{431513\nu}{49766400}+\frac{11639731}{159667200})\nu, \quad
\end{eqnarray*}
\begin{eqnarray*}
\vartheta_{0,12}^{(\nu)}=\frac{7\nu^5}{26542080}+\frac{203\nu^4}{9953280}+\frac{22061\nu^3}{33177600}+\frac{2303\nu^2}{194400}+\frac{6742753\nu}{53222400}+1,
\end{eqnarray*}
and the remaining coefficients that are not mentioned, are zero.\\
Now we provide the features of the new coefficients which are achieved from the rigorous scrutiny of these coefficients and which are in the general state.
\begin{lemma}\label{Thm98}
    The coefficients  $\vartheta_{s,p}^{(\nu)}$ for $s\in \mathbb{Z}$ and $p=2k$, $k\in \mathbb{N}$ in (\ref{eqnarray1804091007}) satisfy:\\
(a) $\vartheta_{0,p}^{(\nu)}\geq1$, $\vartheta_{-s,p}^{(\nu)}=\vartheta_{s,p}^{(\nu)}$ for all $\mid s \mid \geq 1$,\\
(b) $(-1)^s \vartheta_{s,p}^{(\nu)}\geq 0$ for all $s$ (alternately being the positive and negative sentences),\\
(c) $\mid \vartheta_{s,p}^{(\nu)} \mid \geq \mid \vartheta_{s+1,p}^{(\nu)} \mid$ for each even number $p$ and fix positive integer $s$,\\
(d) $\mid \vartheta_{s,p}^{(\nu)} \mid \geq \mid \vartheta_{s,p+2}^{(\nu)} \mid$ for each positive number $s$ and fix even number $p$,\\
(e) $\sum_{s=-\infty}^{\infty}\vartheta_{s,p}^{(\nu)}=1$,\\
(f) For any positive integer $n$ and $m$ with $n<m$, we have $\sum_{s=-m+n}^{n}\vartheta_{s,p}^{(\nu)}>0$.
\end{lemma}
\begin{theorem}\label{Thm95}
\cite{Young1972} \textit{If} A \textit{is a real symmetric matrix with non-negative diagonal elements which is irreducible and has weak diagonal dominance, then} A \textit{is positive definite.}
\end{theorem}
\begin{lemma}\label{Thm97}
The transformation
\begin{eqnarray*}
\varphi(z)=\frac{120-60z+12z^2-z^3}{120+60z+12z^2+z^3}
\end{eqnarray*}
    maps the right half of the complex plane onto inner unit disk.
\end{lemma}
\begin{proof}
We assume $a>0$ for $z=a+ib$, $i=\sqrt{-1}$ and because of size of $z$, we have $a^{2}+b^{2}>0$
and also we are able to write by following manner in \cite{Ding2009600,Rahman2014264}
\begin{equation*}
48a[600+70a^2+(a^2+b^2)^2]>0,
\end{equation*}
to apply a few calculating on recently relation, the outcome is the following inequality
\begin{equation*}
14400(z+\overline{z})+240(z^3+\overline{z}^3)+1440\overline{z}z(z+\overline{z})+24(\overline{z}z)^2(z+\overline{z})>0,
\end{equation*}
then
\begin{equation*}
14400z+240z^3+14400\overline{z}+1440\overline{z}z^2+1440\overline{z}^2z+24\overline{z}^2z^3+
240\overline{z}^3+24\overline{z}^3z^2>0,
\end{equation*}
hence
\begin{equation*}
(120-60\overline{z}+12\overline{z}^2-\overline{z}^3)(120-60z+12z^2-z^3)<(120+60\overline{z}+
12\overline{z}^2+\overline{z}^3)(120+60z+12z^2+z^3),
\end{equation*}
therefore
\begin{equation*}
\mid \frac{120-60z+12z^2-z^3}{120+60z+12z^2+z^3}\mid <1,
\end{equation*}
this shows the prove of the lemma is finished.
\end{proof}

\begin{lemma}\label{Thm96}
\cite{Thomas1995} \textit{Suppose} $M$ \textit{is symmetric, then} $\rho(M)\leq 1+C \Delta t$ \textit{for some non-negative} $C$ \textit{is a necessary and sufficient condition for stability of difference scheme}
\begin{eqnarray*}
U_{k+1}=MU_{k}
\end{eqnarray*}
\textit{with respect to the} $\ell_{2,\Delta x}$ \textit{norm, where} $\rho(M)$ \textit{denotes the spectral radius of the matrix} $M$.
\end{lemma}

\section{Numerical scheme for RFADE}
\label{sec3}
\noindent In this section, we will obtain a new high order finite difference scheme for solving equation (\ref{eqnarray1804090938}) based upon matrix transform method.\\
Here we consider the domain of the problem, which includes space and time direction and we divide it via spatial and temporal step sizes. Let
\begin{eqnarray*}
x_{i}=ih, \quad i=0,1,2,...m, \quad t_{j}=kj, \quad j=0,1,2,...,n,
\end{eqnarray*}
where $h=\frac{L}{m}$ and $k=\frac{T}{n}$ are space and time steps, respectively. The values of the finite difference approximations of $u(x,t)$ at the grid are denoted by
\begin{eqnarray} \label{eqnarray1804091032}
u_{i,j}=u(x_{i},t_{j}).
\end{eqnarray}
Assume that $u(x,t)$ is sufficiently smooth function and replace the fractional partial derivatives stated in (\ref{eqnarray1804090938}) with respect to $x$ by the approximated formula (\ref{eqnarray1804091007})
\begin{eqnarray} \label{eqnarray1804091033}
\frac{\partial^{\alpha}u(x_{i},t)}{\partial
|x|^{\alpha}}=-h^{-\alpha}\sum_{s=-\infty}^{\infty}\vartheta_{s,p}^{(\alpha)}\sum_{k=-\infty}^{\infty}\omega_{k}^{(\alpha)}u(x_{i-k-s},t)+\mathcal{O}(h^{p}),\quad 1< \alpha \leq2 ,
\end{eqnarray}
and
\begin{eqnarray} \label{eqnarray1804091034}
\frac{\partial^{\beta}u(x_{i},t)}{\partial
|x|^{\beta}}=-h^{-\beta}\sum_{s=-\infty}^{\infty}\vartheta_{s,p}^{(\beta)}\sum_{k=-\infty}^{\infty}\omega_{k}^{(\beta)}u(x_{i-k-s},t)+\mathcal{O}(h^{p}),\quad 0< \beta < 1 ,
\end{eqnarray}
Let $u_{i}(t)=u(x_{i},t)$, for $i=1,2,...,m-1$, then the RSFADE (\ref{eqnarray1804090938}) can be cast into the following system of time ordinary differential equations by considering formulas (\ref{eqnarray1804091033}) and (\ref{eqnarray1804091034}) based on mesh sizes in spatial direction.
\begin{eqnarray}\label{eqnarray1804091035}
\frac{\partial u_{i}(t)}{\partial t}=-(\sum_{s=-m+i}^{i}\vartheta_{s,p}^{(\alpha)}\sum_{k=-m+i}^{i}\frac{\mathcal{K}_{\alpha} \omega_{k}^{(\alpha)}}{h^{\alpha}}+\sum_{s=-m+i}^{i}\vartheta_{s,p}^{(\beta)}\sum_{k=-m+i}^{i}\frac{\mathcal{K}_{\beta} \omega_{k}^{(\beta)}}{h^{\beta}})u_{i-k-s}(t),
\end{eqnarray}
Denote
\begin{eqnarray*}
U(t)=[u_{1}(t),u_{2}(t),...,u_{m-1}(t)]^{T},
\end{eqnarray*}
\begin{eqnarray*}
U_{0}=U(0)=[u_{1}(0),u_{2}(0),...,u_{m-1}(0)]^{T},
\end{eqnarray*}
then the equation (\ref{eqnarray1804091035}) can be rewritten as the following matrix form:
\begin{eqnarray}\label{eqnarray1804091036}
\left\{
  \begin{array}{l}
    \frac{dU(t)}{dt}=-(A^{(\alpha)}M^{(\alpha)}+A^{(\beta)}M^{(\beta)})U(t), \\
    U(0)=U_{0}.
  \end{array}
\right.
\end{eqnarray}
in which matrices $A^{(\alpha)}$, $M^{(\alpha)}$, $A^{(\beta)}$ and $M^{(\beta)}$ are defined the following form
\begin{eqnarray*}
A^{(\nu)}_{i,j}=\vartheta_{\mid
i-j\mid,p}^{(\nu)}, \quad \mbox{for} \quad \nu=\alpha,\beta, \quad \mbox{and} \quad i,j=1,2,...,m-1,
\end{eqnarray*}
\begin{eqnarray*}
M^{(\nu)}_{i,j}=\mathcal{K}_{\nu}h^{-\nu} \omega_{\mid
i-j\mid}^{(\nu)},\quad \mbox{for} \quad \nu=\alpha,\beta,  \quad \mbox{and} \quad i,j=1,2,...,m-1.
\end{eqnarray*}
With respect to Lemma \ref{Thm99} and  Lemma \ref{Thm98}, these follow that the diagonal entry of all four matrices are positive and also these matrix are symmetric and strictly diagonally dominant. Therefore, these matrices are symmetric positive definite and so we'll conclude the matrix $A^{(\alpha)}M^{(\alpha)}+A^{(\beta)}M^{(\beta)}$ is also a symmetric positive definite matrix.\\

Let $S=A^{(\alpha)}M^{(\alpha)}+A^{(\beta)}M^{(\beta)}$. Due to being symmetric positive definite of the Matrix $S$, the exact solution of the equation (\ref{eqnarray1804091036}) can be written as follow
\begin{eqnarray*}
U(t)=\exp(-t S)U_{0},
\end{eqnarray*}
By rewriting the exact answer for the time steps $t_{j}$ and $t_{j+1}$, the following formulas are obtained
\begin{eqnarray*}
U(t_{j+1})=\exp(-(j+1)k S)U_{0},
\end{eqnarray*}
and
\begin{eqnarray*}
U(t_{j})=\exp(-(j)k S)U_{0},
\end{eqnarray*}
Thus, we can obtain a recurrence formula for solving equation (\ref{eqnarray1804091036}) by emerging the two last formulas
\begin{eqnarray}\label{eqnarray1804091043}
U(t_{j+1})=\exp(-k S)U(t_{j}).
\end{eqnarray}
Now we approximate $\exp(-z)$ by using $[3,3]$ Pad\'{e} approximation  \cite{Zhang2014109}
\begin{eqnarray}\label{eqnarray202104270146}
\exp(-z)=\frac{120-60z+12z^2-z^3}{120+60z+12z^2+z^3}+\mathcal{O}(z^{7}).
\end{eqnarray}
equation (\ref{eqnarray1804091043}) takes the following form with the aid of above cited Pad\'{e} approximation scheme
\begin{eqnarray}\label{eqnarray1804091044}
U(t_{j+1})=(120+60k S+12(k S)^2+(k S)^3)^{-1}(120-60k S+12(k S)^2-(k S)^3)U(t_{j}).
\end{eqnarray}
According to Formula (\ref{eqnarray202104270146}), the local error in each time step is of the order of seven. Therefore, the total error for time calculations (for $n$ time steps) will be of the order of six.
\section{Stability analysis for numerical method}
\label{sec4}
In this section, we prove the unconditional stability of the numerical method.
\noindent \begin{theorem}
The iterative scheme defined by (\ref{eqnarray1804091044}) to solve the RFADE (\ref{eqnarray1804090938})-(\ref{eqnarray1804090940}) is unconditionally stable.
\end{theorem}
\begin{proof}
According to the being symmetric positive definite of matrix $S$ and positivity of the diagonal entry of the matrix $S$, we consequence based on Theorem \ref{Thm95} that all of the eigenvalues of matrix $S$ have positive real parts.
Let
\begin{eqnarray*}
M=(120+60k S+12(k S)^2+(k S)^3)^{-1}(120-60k S+12(k S)^2-(k S)^3),
\end{eqnarray*}
The spectral radius of the matrix $M$ is given by
\begin{eqnarray*}
\rho(M)=\max \mid \mu_{i} \mid, \quad i=1,2,...,m-1,
\end{eqnarray*}
where $\mu_{i}$ are the eigenvalues of the matrix
\begin{eqnarray*}
(120+60k S+12(k S)^2+(k S)^3)^{-1}(120-60k S+12(k S)^2-(k S)^3).
\end{eqnarray*}
It is easy to conclude that the eigenvalues of the matrix \cite{Hoffman1961}
\begin{eqnarray*}
(120+60k S+12(k S)^2+(k S)^3)^{-1}(120-60k S+12(k S)^2-(k S)^3),
\end{eqnarray*}
are given by
\begin{eqnarray*}
\mu_{i}=\frac{120-60k\lambda_{i}(S)+12k^2(\lambda_{i}(S))^2-k^3(\lambda_{i}(S))^3}{120+60k\lambda_{i}(S)+12k^2(\lambda_{i}(S))^2+k^3(\lambda_{i}(S))^3},\quad i=1,2,...,m-1,
\end{eqnarray*}
by using the Lemma \ref{Thm97} and being positive of the $\lambda_{i}(S)$, we have
\begin{eqnarray*}
\mid \mu_{i} \mid < 1, \quad i=1,2,...,m-1.
\end{eqnarray*}
It follows from above mentioned inequality that the spectral radius of matrix $M$ is smaller than 1. Hence, based on the Lemma \ref{Thm96} the iterative scheme (\ref{eqnarray1804091044}) is unconditionally stable.
\end{proof}
\section{Numerical examples of the RFADE}
\label{sec5}
\noindent
In order to verify the validity of the theoretical topics mentioned in this article, the maximum error and the approximate convergence rate are considered for the following two examples. For this purpose, the necessary formulas have been defined.
We consider $L_{2}$ norm resembling as the device for error between the analytical solution and approximate solution showing via $e(h,k)$. The experimental convergence order in spatial direction $\mathcal{R}_{s}(h,k)$ computed by the formula
\begin{eqnarray*}
\mathcal{R}_{s}(h,k)={\log \frac{e(2h,k)}{e(h,k)}}/{\log 2}, \quad \mbox{for small values}\quad k
\end{eqnarray*}
and another one in temporal direction $\mathcal{R}_{t}(h,k)$ calculated by the formula
\begin{eqnarray*}
\mathcal{R}_{t}(h,k)={\log \frac{e(h,2k)}{e(h,k)}}/{\log 2}, \quad \mbox{for small values}\quad h
\end{eqnarray*}

\begin{example}
\label{exp:1}
Consider the following Riesz space fractional advection-dispersion equation
\begin{eqnarray}\label{eqnarray1804091055}
\frac{\partial u(x,t)}{\partial
t}=\mathcal{K}_{\alpha}\frac{\partial^{\alpha}u(x,t)}{\partial
|x|^{\alpha}}+\mathcal{K}_{\beta}\frac{\partial^{\beta}u(x,t)}{\partial
|x|^{\beta}},\quad 0<x<\pi,\quad t>0,
\end{eqnarray}
associated with the initial value and zero Dirichlet boundary conditions
\begin{eqnarray}\label{eqnarray1804091056}
u(x,0)=x^2(\pi-x),\quad 0<x<\pi,
\end{eqnarray}
\begin{eqnarray}\label{eqnarray1804091057}
u(0,t)=u(\pi,t)=0,\quad 0\leq t\leq T.
\end{eqnarray}
With using separation of variables method, we obtained eigenvalue $\lambda_{n}=n$ and eigenfunction $\varphi_{n}(x)=sin(n x)$ for $n=1,2,...$ due to Dirichlet boundary conditions and hence the analytical solution of equations (\ref{eqnarray1804091055})-(\ref{eqnarray1804091057}) would be at following formula:
\begin{eqnarray*}
u(x,t)=\sum_{n=1}^{\infty}[\frac{8}{n^{3}}(-1)^{n+1}-\frac{4}{n^{3}}]sin(n x)\exp(-[\mathcal{K}_{\alpha}(n^{2})^{\frac{\alpha}{2}}+\mathcal{K}_{\beta}(n^{2})^{\frac{\beta}{2}}]t),
\end{eqnarray*}
In this example, the approximation formulas (\ref{eqnarray1804091033}) and (\ref{eqnarray1804091034}) for the $p=6$ are used for the approximation of the Riesz fractional derivative of order $\alpha$ and $\beta$, respectively, and the $[3,3]$ Pad\'{e} approximation is presented for the solution of the time ODE obtained from the previous approximation. Here, $\mathcal{K}_{\alpha}=\mathcal{K}_{\beta}=0.25$.
\begin{table}[htbp]
\caption{\label{tab:1}Maximum errors and corresponding rates for solving RSFADE (\ref{eqnarray1804091055})-(\ref{eqnarray1804091057}) with $\mathit{k}= 0.001$}
\centering
\fbox{%
\begin{tabular}{c|cc|cc}
& \multicolumn{2}{c|}{$\alpha=1.8$, $\beta=0.9$} & \multicolumn{2}{c}{$\alpha=1.6$, $\beta=0.7$} \\
$\mathit{h}$ & Max Error & Error Rate & Max Error & Error Rate\\ \hline
$0.10000\pi$ & $9.10145E-04$ & - & $8.44135E-04$ & - \\
$0.05000\pi$ & $1.91731E-05$ & 5.56894 & $1.83810E-05$ & 5.52119 \\
$0.02500\pi$ & $3.99604E-07$ & 5.58437 & $3.66746E-07$ & 5.64729 \\
$0.01250\pi$ & $7.83920E-09$ & 5.67172 & $7.02208E-09$ & 5.70674 \\
$0.00625\pi$ & $3.26552E-11$ & 5.81562 & $1.26130E-10$ & 5.79892
\end{tabular}}
\end{table}
Table \ref{tab:1}. shows the magnitude of the maximum error and estimated convergence order, at time $t = 1.0$, between the analytical solution and the numerical solution obtained by the proposed method for solving RSFADE (\ref{eqnarray1804091055})-(\ref{eqnarray1804091057}) in which considered different order values such as $\alpha=1.8$, $\beta=0.9$ and $\alpha=1.6$, $\beta=0.7$ with halved spatial step sizes and $k=0.001$. Table \ref{tab:1}. demonstrates that the convergence rate of error in the spatial direction is approximately equal to sixth.

\begin{table}[htbp]
\caption{\label{tab:2}Maximum errors and corresponding rates for solving RSFADE (\ref{eqnarray1804091055})-(\ref{eqnarray1804091057}) with $\mathit{h}= 0.001\pi$}
\centering
\fbox{%
\begin{tabular}{c|cc|cc}
& \multicolumn{2}{c|}{$\alpha=1.8$, $\beta=0.9$} & \multicolumn{2}{c}{$\alpha=1.6$, $\beta=0.7$} \\
$\mathit{k}$ & Max Error & Error Rate & Max Error & Error Rate\\ \hline
$0.10000$ & $4.62484E-05$ &    -    & $3.52977E-05$ &    -      \\
$0.05000$ & $9.67009E-07$ & 5.57973 & $7.59453E-07$ &  5.53847  \\
$0.02500$ & $1.93074E-08$ & 5.64630 & $1.54910E-08$ &  5.61546  \\
$0.01250$ & $3.69215E-10$ & 5.70855 & $3.10036E-10$ &  5.64285  \\
$0.00625$ & $6.45402E-12$ & 5.83812 & $5.65678E-12$ &  5.77631
\end{tabular}}
\end{table}

Table \ref{tab:2}. displays the magnitude of the maximum error and estimated convergence order, at time $t = 1.0$, between the analytical solution and the numerical solution obtained by the proposed method for solving RSFADE (\ref{eqnarray1804091055})-(\ref{eqnarray1804091057}), where $\alpha$ and $\beta$ pairs are corresponding to two distinct values: $\alpha=1.8$, $\beta=0.9$ and $\alpha=1.6$, $\beta=0.7$ with halved temporal step sizes and $h=0.001\pi$. Table \ref{tab:2}. demonstrates that the convergence rate of error in the temporal direction is approximately equal to sixth.

\begin{figure}
\begin{center}
\mbox{\includegraphics[height=2.85in]{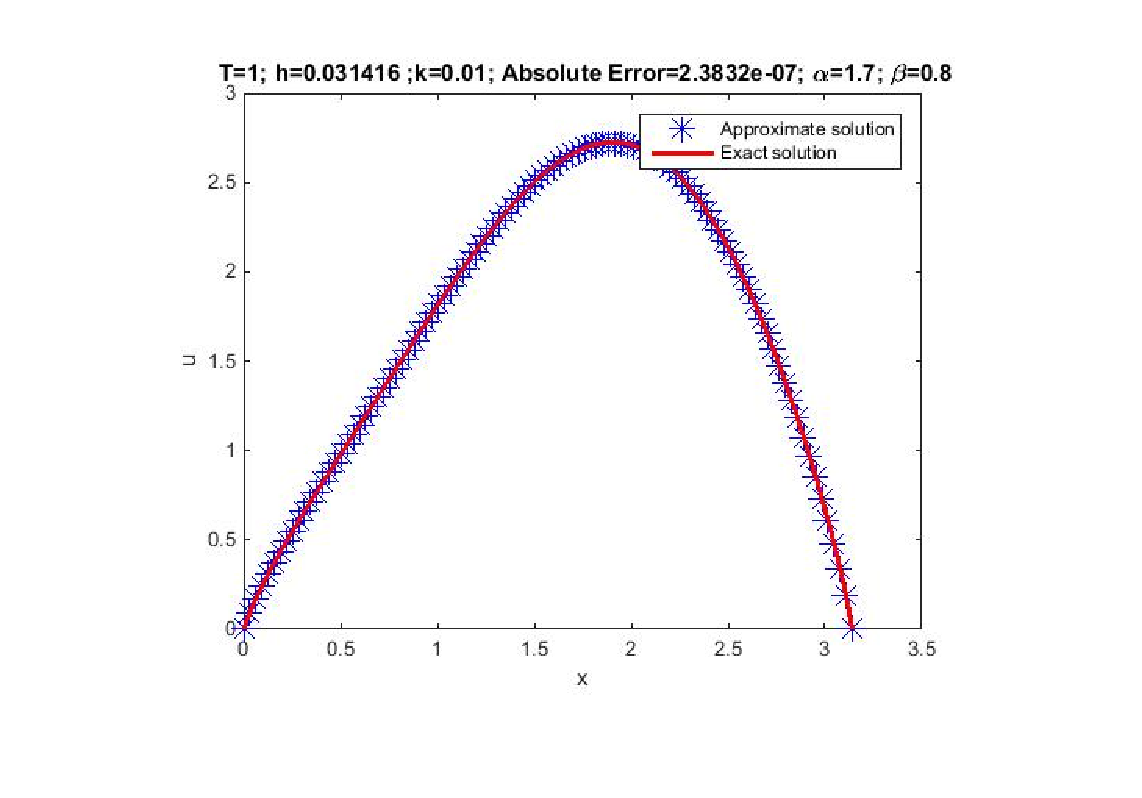}}
\end{center}
\caption{The comparison of the numerical solution and analytical solution of RSFADE (\ref{eqnarray1804091055})-(\ref{eqnarray1804091057})}
\label{fig:5-1-2}
\end{figure}

Figure \ref{fig:5-1-2} presents the comparison of the numerical solution obtained by mixed high order fractional centered difference scheme with $[3,3]$ Pad\'{e} approximation method and analytical solution with $h=\frac{\pi}{100}, k=0.01$ for $\alpha=1.7$ and $\beta=0.8$. The graph corresponding to the approximate solution and the analytical solution is overlapping.\\

\begin{figure}
\begin{center}
\mbox{\includegraphics[height=1.77in]{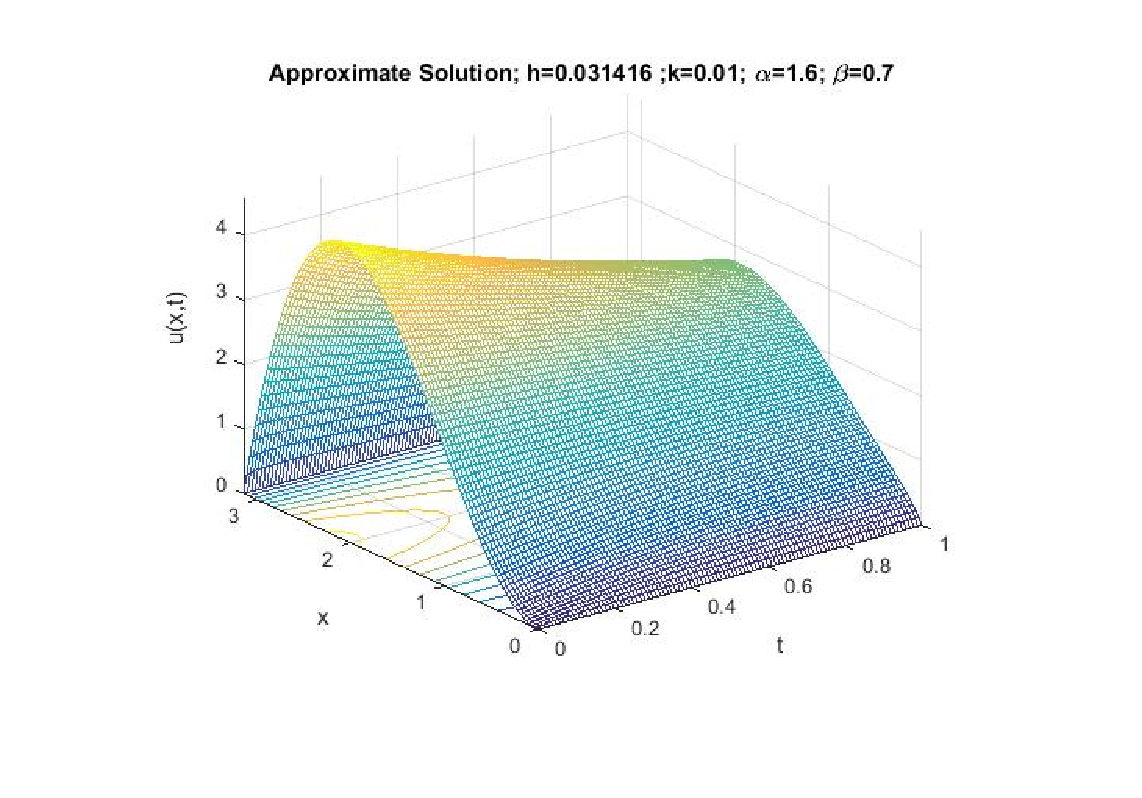}}
\mbox{\includegraphics[height=1.77in]{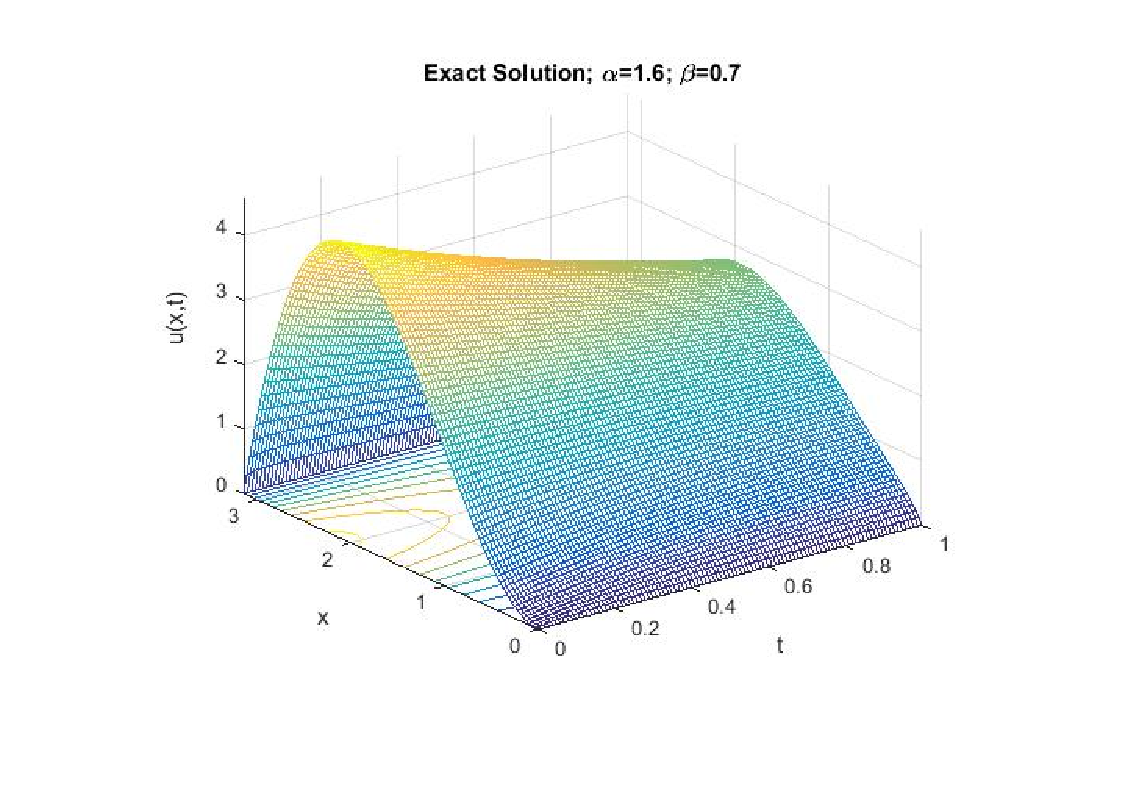}}
\end{center}
\caption{3D plot of the numerical solution and analytical solution for RSFADE (\ref{eqnarray1804091055})-(\ref{eqnarray1804091057})}
\label{fig:5-1-3}
\end{figure}

Figure \ref{fig:5-1-3} displays the approximate solution (Left) and analytical solution (Right) profiles of RSFADE (\ref{eqnarray1804091055})-(\ref{eqnarray1804091057}) over space for $0<t<1$ when $\alpha=1.6$, $\beta=0.7$. Figures \ref{fig:5-1-2} and \ref{fig:5-1-3}, and Tables \ref{tab:1} and \ref{tab:2} illustrate that the numerical experiments are in excellent agreement with analytical results.
\end{example}

\begin{example}
\label{exp:2}
Consider the following Riesz space fractional advection-dispersion equation
\begin{eqnarray}\label{eqnarray1804091134}
\frac{\partial u(x,t)}{\partial
t}=\mathcal{K}_{\alpha}\frac{\partial^{\alpha}u(x,t)}{\partial
|x|^{\alpha}}+\mathcal{K}_{\beta}\frac{\partial^{\beta}u(x,t)}{\partial
|x|^{\beta}},\quad 0<x<1,\quad t>0,
\end{eqnarray}
associated with the initial value and zero Dirichlet boundary conditions
\begin{eqnarray}\label{eqnarray1804091135}
u(x,0)=x(1-x),\quad 0<x<1,
\end{eqnarray}
\begin{eqnarray}\label{eqnarray1804091136}
u(0,t)=u(1,t)=0,\quad 0\leq t\leq T.
\end{eqnarray}
With using separation of variables method, we reached eigenvalues $\lambda_{n}=(2n-1)\pi$ and eigenfunctions $\varphi_{n}(x)=sin((2n-1)\pi x)$ for $n=1,2,...$ because of Dirichlet boundary conditions and hence the analytical solution of equations (\ref{eqnarray1804091134})-(\ref{eqnarray1804091136}) would be following the series:
\begin{eqnarray*}
u(x,t)=\sum_{n=1}^{\infty}\frac{8}{((2n-1)\pi)^{3}}sin((2n-1)\pi x)\exp(-[\mathcal{K}_{\alpha}(((2n-1)\pi)^{2})^{\frac{\alpha}{2}}+\mathcal{K}_{\beta}(((2n-1)\pi)^{2})^{\frac{\beta}{2}}]t),
\end{eqnarray*}
In this example, the approximation formulas (\ref{eqnarray1804091033}) and (\ref{eqnarray1804091034}) for the $p=4$ are applied for the approximation of the Riesz fractional derivative of order $\alpha$ and $\beta$ respectively, and the $[3,3]$ Pad\'{e} approximation is considered for the solution of the time ODE obtained from the previous approximation. Here, $\mathcal{K}_{\alpha}=\mathcal{K}_{\beta}=0.25$.
\begin{table}[htbp]
\caption{\label{tab:3}Maximum errors and corresponding rates for solving RSFADE (\ref{eqnarray1804091134})-(\ref{eqnarray1804091136}) with $\mathit{k}= 0.001$}
\centering
\fbox{%
\begin{tabular}{c|cc|cc}
& \multicolumn{2}{c|}{$\alpha=1.8$, $\beta=0.9$} & \multicolumn{2}{c|}{$\alpha=1.7$, $\beta=0.8$} \\
$\mathit{h}$ & Max Error & Error Rate & Max Error & Error Rate\\ \hline
$0.10000$ & $1.39882E-03$ &    -    & $1.44234E-03$ &    -     \\
$0.05000$ & $1.11466E-04$ & 3.64953 & $1.31934E-04$ & 3.45052  \\
$0.02500$ & $8.44864E-06$ & 3.72174 & $1.15595E-05$ & 3.51267  \\
$0.01250$ & $5.74142E-07$ & 3.87924 & $8.58159E-07$ & 3.75169  \\
$0.00625$ & $3.76350E-08$ & 3.93126 & $5.80712E-08$ & 3.88535
\end{tabular}}
\end{table}

Table \ref{tab:3}. gives details of the maximum error and approximated convergence order, at time $t = 1.0$, between the analytical solution and the numerical solution obtained by the suggested method for solving RSFADE (\ref{eqnarray1804091134})-(\ref{eqnarray1804091136}) for different order values $\alpha=1.8$, $\beta=0.9$ and $\alpha=1.7$, $\beta=0.8$ with halved spatial step sizes and $k=0.001$. Table \ref{tab:3}. confirms that the convergence rate of error in the spatial direction is numerically equal to fourth.\\

\begin{table}[htbp]
\caption{\label{tab:4}Maximum errors and corresponding rates for solving RSFADE (\ref{eqnarray1804091134})-(\ref{eqnarray1804091136}) with $\mathit{h}= 0.001$}
\centering
\fbox{%
\begin{tabular}{c|cc|cc}
& \multicolumn{2}{c|}{$\alpha=1.8$, $\beta=0.9$} & \multicolumn{2}{c|}{$\alpha=1.7$, $\beta=0.8$} \\
$\mathit{k}$ & Max Error & Error Rate & Max Error & Error Rate\\ \hline
$0.10000$ & $4.82904E-04$ &    -    & $4.66929E-04$ &    -     \\
$0.05000$ & $9.69414E-06$ & 5.63848 & $1.03693E-05$ & 5.49281  \\
$0.02500$ & $1.84807E-07$ & 5.71302 & $2.04733E-07$ & 5.66243  \\
$0.01250$ & $3.32526E-09$ & 5.79641 & $3.93496E-09$ & 5.70125  \\
$0.00625$ & $5.48630E-11$ & 5.92149 & $6.87941E-11$ & 5.83792
\end{tabular}}
\end{table}

Table \ref{tab:4}. exhibits the maximum error and experimented convergence rate, at time $t = 1.0$, between the analytical solution and the numerical solution obtained by the proposed method for solving RSFADE (\ref{eqnarray1804091134})-(\ref{eqnarray1804091136}), where $\alpha$ and $\beta$ pairs are corresponding to two distinct values: $\alpha=1.8$, $\beta=0.9$ and $\alpha=1.7$, $\beta=0.8$ with halved temporal step sizes and $h=0.001$. Table \ref{tab:4}. verifies that the convergence rate of error in the temporal direction is approximately equal to sixth.

\begin{figure}
\begin{center}
\mbox{\includegraphics[height=2.85in]{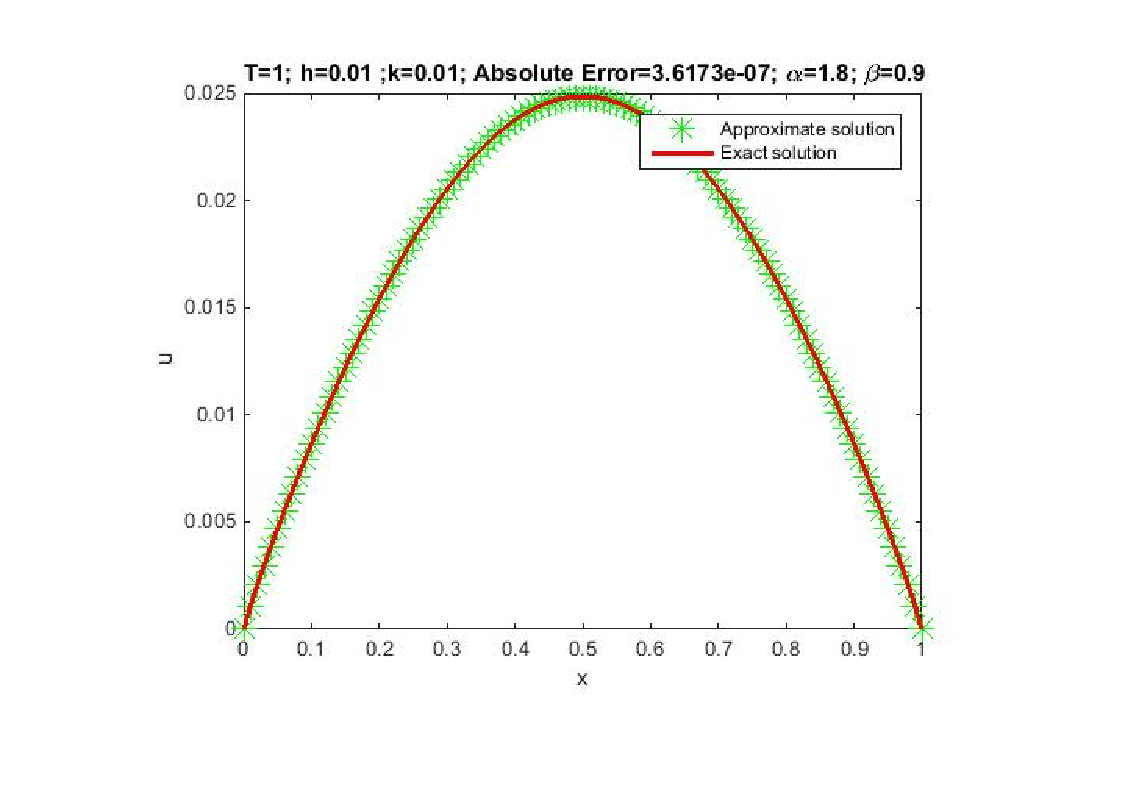}}
\end{center}
\caption{The comparison of the numerical solution and analytical solution of RSFADE (\ref{eqnarray1804091134})-(\ref{eqnarray1804091136})}
\label{fig:5-2-1}
\end{figure}

Figure \ref{fig:5-2-1} depicts the comparison of the numerical solution obtained using mixed high order fractional centered difference scheme with $[3,3]$ Pad\'{e} approximation method and analytical solution with $h=0.01, k=0.01$ for $\alpha=1.8$ and $\beta=0.9$. The approximate solutions obtained by the proposed method are consistent with the analytical solutions.

\begin{figure}
\begin{center}
\mbox{\includegraphics[height=1.77in]{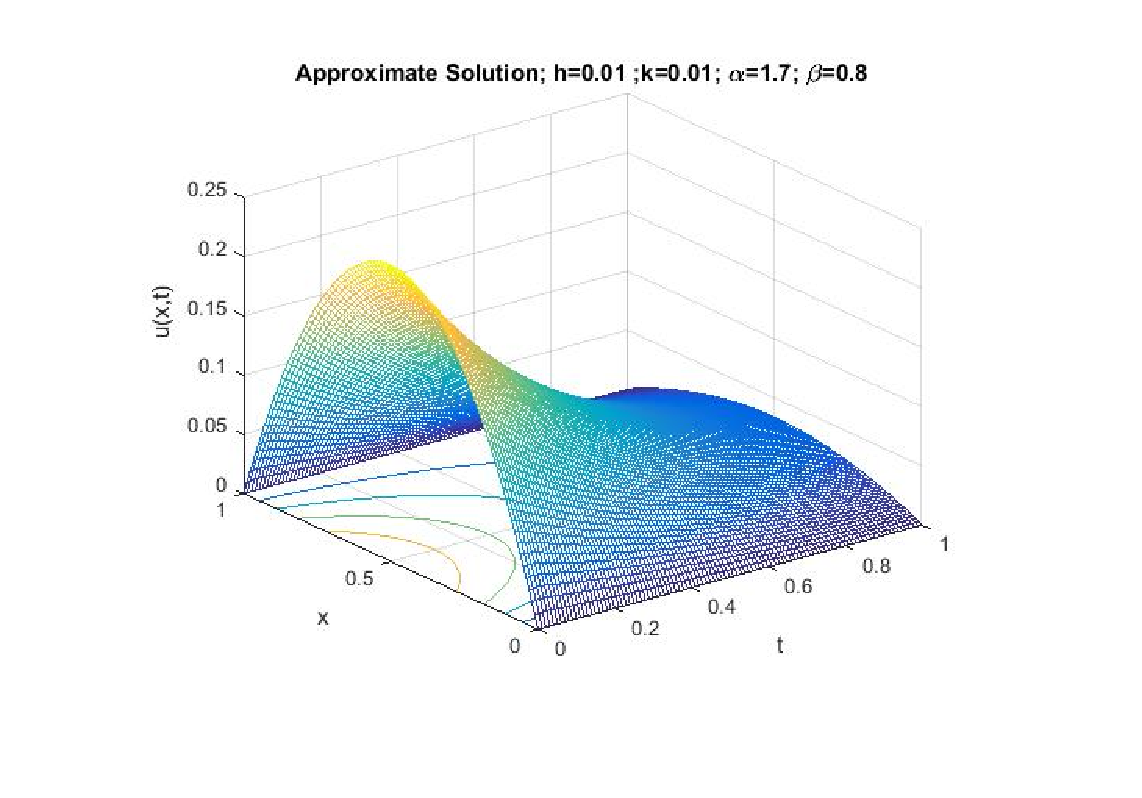}}
\mbox{\includegraphics[height=1.77in]{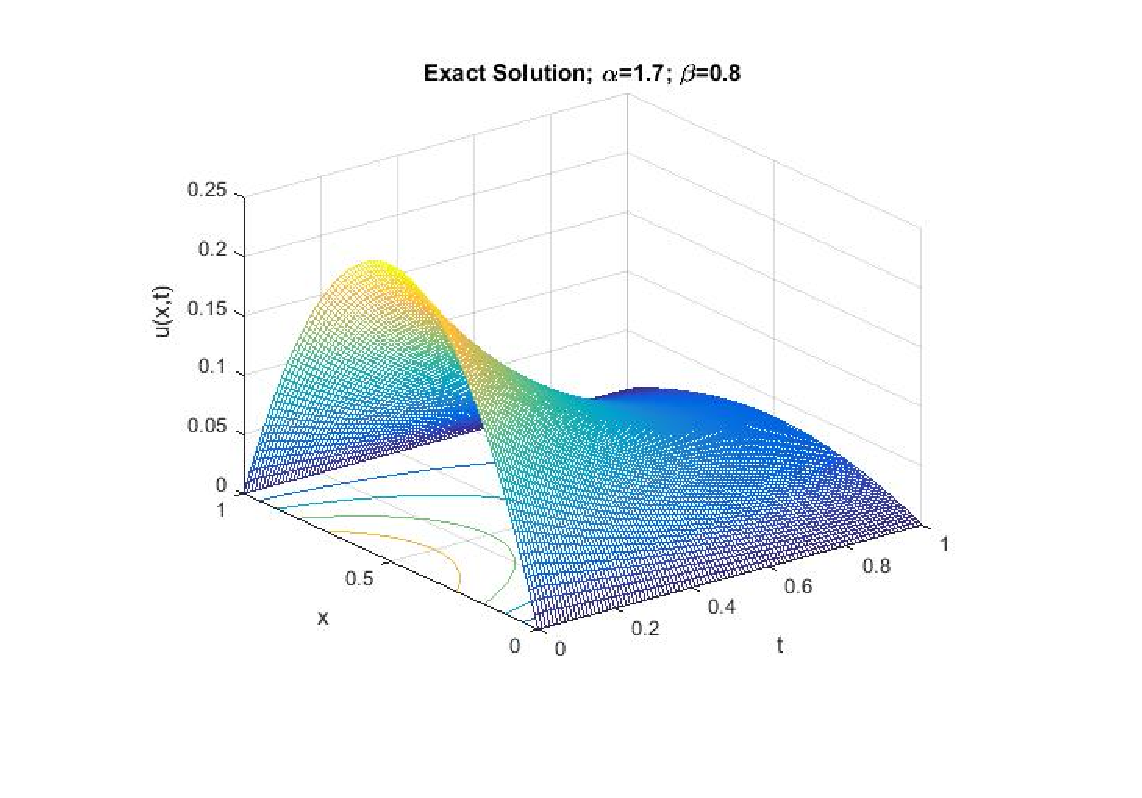}}
\end{center}
\caption{3D plot of the numerical solution and analytical solution for RSFADE (\ref{eqnarray1804091134})-(\ref{eqnarray1804091136})}
\label{fig:5-2-3}
\end{figure}

Figure \ref{fig:5-2-3} displays the approximate solution (Left) and analytical solution (Right) profiles of RSFADE (\ref{eqnarray1804091134})-(\ref{eqnarray1804091136}) over space for $0<t<1$ when $\alpha=1.7$, $\beta=0.8$. Figures \ref{fig:5-2-1} and \ref{fig:5-2-3}, and Tables \ref{tab:3} and \ref{tab:4} exemplify the validity and accuracy of the numerically proposed method.
\end{example}
\section{Conclusions}
In this paper, we considered a novel finite difference method based on a fractional centered difference formula from an arbitrary order for the discretisation of the Riesz fractional derivatives coupled with a $[3,3]$ Pad\'{e} approximation method for time stepping strategy for the numerical solution of the Riesz fractional advection-dispersion equation in a finite domain with an initial value and homogeneous Dirichlet boundary conditions. It is proved that the proposed method is unconditionally stable in view of the matrix analysis method. Numerical results obtained from solving the Riesz fractional advection-dispersion equation demonstrates the theoretical results and verify the efficiency of the proposed method.


\section*{Acknowledgments}
The authors express their deep gratitude to the Research Council of University of Mohaghegh Ardabili for funding this study.


\end{document}